\documentclass[12pt]{amsart}
\usepackage{amscd}

\newtheorem{theorem}{Theorem}
\newtheorem{lemma}[theorem]{Lemma}
\newtheorem{corollary}{Corollary}
\newtheorem{proposition}[theorem]{Proposition}

\theoremstyle{definition}

\newtheorem{notation}{Notation}

\theoremstyle{remark}
\newtheorem*{remark}{Remark}

\usepackage{amsfonts}
\newcommand{\field}[1]{\mathbb{#1}}
\newcommand{\Q}{\field{Q}}
\newcommand{\R}{\field{R}}
\newcommand{\Z}{\field{Z}}

\newcommand{\A}{\field{A}}

\renewcommand{\a}{\alpha}
\renewcommand{\b}{\beta}

\newcommand{\D}{\Delta}

\newcommand{\G}{\Gamma}
\renewcommand{\k}{\kappa}
\renewcommand{\l}{\lambda}

\newcommand{\fg}{\mathfrak g}
\newcommand{\fp}{\mathfrak p}

\newcommand{\fu}{\mathfrak u}
\newcommand{\fv}{\mathfrak v}
\newcommand{\fl}{\mathfrak l}

\newcommand{\ra}{\rightarrow}

\begin{document}

\title[Centrality]{Centrality}

\author{T. N. Venkataramana}

\address{ School of Mathematics, Tata Institute of Fundamental 
Research, Homi Bhabha Road, Bombay - 400 005, INDIA.}

\email{venky@math.tifr.res.in}

\subjclass{Primary 20H05 Secondary 22E40\\
T. N. Venkataramana,
School of Mathematics, Tata Institute of Fundamental 
Research, Homi Bhabha Road, Bombay - 400 005, INDIA.
venky@math.tifr.res.in}

\date{}

\begin{abstract}

We give  a simple proof of  the centrality of the  congruence subgroup
kernel in the higher rank isotropic case.

\end{abstract}

\maketitle

\section{Introduction}

In this paper, we give a simple  proof of the well known centrality of
the congruence subgroup  kernel in the ``higher  rank'' isotropic case
(\cite{R1}, \cite{R2}). That  is, we prove the following  (we refer to
Section 1 for definitions of the terms involved). 

\begin{theorem} \label{maintheorem} The congruence subgroup kernel $C$
associated  to   a  connected  $\Q$-simple  simply   connected  linear
algebraic group  $G$ defined  over $\Q$ with  $\Q-rank (G)\geq  1$ and
$\R-rank (G)\geq  2$, is  {\bf central}  in the  arithmetic completion
$\widehat G$ of $G(\Q)$.
\end{theorem}

The newest part of the proof is  in the case when $\Q -rank (G)=1$ and
the group of integer points $L(\Z)$ of the Levi of a minimal parabolic
$\Q$-subgroup $P$  is a  virtually abelian  infinite group.  To handle
this case, we make use of the following result, which is quite general
and may be of independent interest. The proof uses Dirichlet's theorem
on primes in arithmetic progressions.

\begin{theorem} \label{dirichletlinear} Let $\D \subset GL_n(\Z)$ be a
subgroup. There  exists an integer  $g(n)$ dependent only on  $n$ such
that for any two co-prime integers  $a,b$ and any fixed integer $N$ the
group  $\D  _{a,b}$  generated  by the  congruence  subgroups  $\{  \D
((a+bx)^N) :=  \D \cap GL_n((a+bx)^N\Z):  x\in \Z\}$, (is  normal and)
the exponent of the quotient group $\D/\D _{a,b}$ is bounded by $g(n)$
(i.e. depends only on $n$ and not on the integers $a,b,~and ~N$).
\end{theorem}

Theorem  \ref{dirichletlinear} is  proved  in section  4. Using  this,
Theorem \ref{maintheorem} is proved in section 5 in the case $\Q -rank
(G)=1$ and  the group $L(\Z)$  is infinite and virtually  abelian. The
case when  $L(\Z)$ is not  virtually abelian  is proved in  section 6.
The other sections establish some preliminaries.

\begin{remark}    In    the    generality   stated    here,    Theorem
\ref{maintheorem} is due to Raghunathan; his proof was quite different
in the two  cases (1) $\Q-rank (G)\geq 2$ \cite{R1},  and (2) $\Q-rank
(G)=1,  \quad   \R-rank  (G)\geq  2$  \cite{R2}.    For  earlier  work
(especially in the  case $SL_n$ and $Sp_{2n}$ \cite{BMS})  we refer to
the  bibliography  in  \cite{R1}.  [At the  time  that  \cite{R1}  and
\cite{R2}  were   written,  the  Kneser-Tits  problem  (we  refer  to
\cite{Gille} for the  statement over number fields  and for references
to previous  work) had not  been completely resolved  and consequently
the  formulation  of  Raghunathan  was slightly  different].   In  the
present paper  we give a different  proof especially in the  case when
$\Q-rank (G)=1$ ( and the group of integer points of the Levi subgroup
of a minimal parabolic $\Q$-subgroup is virtually abelian).
\end{remark} 

\begin{remark} To prove  the centrality in the $\Q $  -rank one case (
with integral points of the Levi being virtually abelian), Raghunathan
uses   the    centrality   (proved    by   Serre    \cite{Ser})   when
$G=R_{K/\Q}(SL_2)$ is the Weil restriction of scalars from $K$ to $\Q$
of $SL_2$, where $K$ a number  field having infinitely many units. The
proof  of centrality  in \cite{Ser}  makes  crucial use  of the  Artin
reciprocity  law. In  contrast,  we use  only  Dirichlet's theorem  on
infinitude of primes in arithmetic progressions. Thus the proof is new
even  in  the case  considered  by  Serre  (we  note that  Prasad  and
Rapinchuk \cite{Pr-Ra2} also proved  centrality; the present paper has
considerable  overlap with  \cite{Pr-Ra2}  and \cite{Pr-Ra2}  contains
much  more than  what is  proved in  the present  paper. However,  the
method  of  proof  in  our  paper is  quite  different  and  it  seems
worthwhile to record the proof here).
\end{remark}

\begin{remark}  Raghunathan  first  proved centrality  for  the  group
$SU(2,1)$: suppose  $L/K$ is  a quadratic  extension of  number fields
with  $K\neq \Q$,  $h$ is  an  isotropic non-degenerate  form in  three
variables defined on $L^3$ and hermitian with respect to $L/K$, $G$ is
the unit group  of $h$ i.e.  $G=R_{K/\Q}(SU(2,1)$  where $R_{K/\Q}$ is
the Weil restriction of scalars.  The  proof of this also makes use of
Artin Reciprocity.  In the present  paper, we completely avoid the use
of $SU(2,1)$.   This is  especially important  since, in  \cite{R2}, a
deep theorem on the embedding of suitable $SL_2$ and $SU(2,1)$ is used
(which in turn is based on  the classification of rank one groups over
number fields) to obtain the centrality in the general case from these
two cases.  In contrast, we (by appealing to a theorem of Jacobson and
Morozov) use only a suitably  embedded $SL_2$ (and Dirichlet's theorem
on  primes  in  arithmetic  progressions)  to  get  the  general  case
directly.  In particular, we get another proof for $SU(2,1)$ as well.
\end{remark}

\begin{remark}  Most importantly,  Raghunathan  \cite{R1} proved  that
once the  congruence subgroup kernel  $C$ is  central for any  $G$ (no
assumptions on  $\Q$-rank), then $C$  is {\it finite}; a  much simpler
proof  of  the   finiteness  was  later  given  by   Gopal  Prasad  in
\cite{P}. Moreover, once $C$ is central,  $C$ can be computed: it is a
precisely determined  subgroup of  the group  of roots  of unity  in a
number field  $K$ ($K$  is such that  $G=R_{K/\Q}(\mathcal G)$  is the
Weil  restriction of  scalars of  an absolutely  almost simple  simply
connected group $\mathcal G$ defined over $K$).  Important progress on
the  computation  of  (the  central) $C$  was  made  in  \cite{Pr-R1},
\cite{Pr-R2}, and  its complete determination  was done by  Prasad and
Rapinchuk in \cite{Pr-Ra}. \\

Our paper  does not deal with  this question at all,  and is concerned
only  with the  centrality of  the congruence  subgroup kernel  in the
cases considered in the theorem.
\end{remark}

\begin{remark}  The proof  here works  only for  arithmetic groups  in
characteristic zero but the proof of centrality in \cite{R1},\cite{R2}
works for  all global fields  and for $S$-arithmetic groups;  it seems
possible to adapt  the present proof to the S-arithmetic  case, but it
appears to be rather long.
\end{remark}

{\bf Acknowledgments:}  I learnt  most of  the mathematics  around the
congruence subgroup  property at TIFR, from  Raghunathan, Prasad, Nori
and my  fellow student  Sury. I  am especially  grateful to  Gopal for
suggesting many  topics to work  on, patiently answering  my questions
and going  through my  early papers  (and righting  so many  wrongs in
those papers).  It is a pleasure to  dedicate this paper to him on his
75th birthday. \\

This paper  owes a great deal  to the methods of  \cite{R1}, \cite{R2}
and  I  am  grateful  to  Raghunathan  for  explaining  his  proof  in
detail. \\

I am grateful to the organisers  for their invitation to contribute to
this volume honouring  Gopal. I am especially grateful  to the referee
for a very careful reading of the manuscript and for pointing out some
egregious   errors,  and   for  his   suggestions  on   improving  the
manuscript. \\

A sizeable part of the work was  done during the author's visit to Max
Planck  Institut, Bonn,  Germany during  2019-2020. I  thank MPIM  for
hospitality and  support.  The support  of JC Bose fellowship  for the
period 2020-2025 is gratefully acknowledged.

\newpage 

\section{Generalities}

\subsection{The Congruence Subgroup Kernel}

The  following definitions  and  observations are  well  known and  we
recall them without reference.\\

Let  $G$ be  a  linear  algebraic group  defined  over  $\Q$.  Fix  an
embedding $G\subset  SL_n$ defined  over $\Q$ and  call a  subgroup of
finite index in $G(\Z)=G\cap SL_n(\Z)$ an {\it arithmetic subgroup} of
$G(\Q)$.  We assume  that $G(\Z)$ is Zariski dense in  $G$.  Denote by
$SL_n(m\Z)$ the kernel to  the natural map $SL_n(\Z)\ra SL_n(\Z/m\Z)$;
a  subgroup of  $G(\Z)$ which  contains $G(m\Z):=G\cap  SL_n(m\Z)$ for
some non-zero  integer $m$  is called a  {\it congruence  subgroup} of
$G(\Q)$.  The notion of  arithmetic subgroups and congruence subgroups
of  $G(\Q)$ does  not depend  on  the specific  $\Q$-embedding $G  \ra
SL_n$.  \\

There  is  a  topological  group  structure  on  $G(\Q)$  obtained  by
designating  a fundamental  system  of neighbourhoods  of identity  in
$G(\Q)$ to be arithmetic groups (respectively congruence subgroups) in
$G(\Z)$; this is called the  {\it arithmetic} (resp. {\it congruence})
topology  on $G(\Q)$);  the  group $G(\Q)$  admits  a completion  with
respect to this  topology; this is the  {\bf arithmetic} (respectively
{\bf congruence})  completion of $G(\Q)$, denoted  $\widehat{G}$ (resp
$\overline  {G}$).  These  completions  are  locally compact,  totally
disconnected and Hausdorff. Since the arithmetic topology on $G(\Q)$ is
finer   than  the   congruence  topology,   we  have   a  homomorphism
$\widehat{G} \ra \overline {G}$ of  completions of $G(\Q)$ with kernel
$C$, say. We have an exact sequence of topological groups
\[ 1 \ra C \ra \widehat{G} \ra \overline{G} \ra 1,\] and $C$ is called
the {\bf congruence  subgroup kernel}. The exact  sequence splits over
$G(\Q)$. \\

The restriction of the arithmetic  topology on $G(\Q)$ to the subgroup
$\G= G(\Z)$ is just the profinite  topology on $G(\Z)$; the closure of
$G(\Z)$   in  $\widehat   G$  is   simply  the   profinite  completion
$\widehat{\G}$ of $G(\Z)$; the closure of $G(\Z)$ in $\overline{G}$ is
the congruence completion $\overline{\G}$ of $\G$; it is not difficult
to  see that  $C  \subset \widehat{\G}$  and that  we  have the  exact
sequence
\[ 1 \ra  C \ra \widehat{\G}\ra \overline{\G} \ra 1,\]
which shows that $C$ is a compact profinite group. 

If $G$ is  a unipotent algebraic group, then every  subgroup of finite
index  in $G(\Z)$  is a  congruence subgroup  and $C$  is trivial.  In
particular, if $U\subset G$ is a unipotent algebraic $\Q$-subgroup and
$\G  \subset G(\Q)$  is  an  arithmetic group,  then  there exists  an
integer $m\neq 0$ such that $U(m\Z)\subset \G$. \\

Suppose $\G=G(\Z)$ and  $\{\G_m \subset \G\}$ a  ``cofinal'' family of
arithmetic  groups which  are normal  in $\G$  (i.e.  such  that every
arithmetic subgroup of $\G$ contains  some $\G _m$). Since the closure
of  $\G$  in the  congruence  completion  $\overline{G}$ is  open,  it
follows that the closure $\overline{\G  _m}$ in $\overline{G}$ is also
open  for every  $\G _m$;  hence the  intersection $Cl(\G  _m)=\G \cap
\overline{\G _m}$ is a congruence  subgroup of $\G$ containing $\G_m$:
it is  called the congruence closure  of $\G _m$. It  follows from the
definitions that $C$ is the inverse limit of $Cl(\G _m)/\G_ m$:
\begin{equation} \label{Casinverselimit} 
C=\varprojlim Cl(\G _m)/\G _m.
\end{equation}

We will say that  relative to a finite set $S$ of  primes, and a fixed
integer $e$, an  integer $m$ is sufficiently deep if  $m$ is divisible
by all the powers $p^e$ for all the primes $p$ in $S$.  In the inverse
limit for $C$,  we may only choose $m$ sufficiently  deep relative to
any  finite  subset  $S$  of  primes. We  will  use  this  observation
repeatedly in the sequel.\\

If $\theta: H  \ra G$ is a morphism of  algebraic $\Q$-groups, then it
induces   maps  $\widehat{H}\ra   \widehat{G}$  and   $\overline{H}\ra
\overline{G}$ and  also a  homomorphism $\theta: C_H  \ra C_G$  of the
corresponding congruence subgroup kernels. \\

If  $G(\Z)$ is  Zariski dense  in  $G$ as  before, and  $G$ is  simply
connected, then  {\it strong  approximation} holds and  the congruence
completion $\overline{G}$ of $G$ is just the group $G(\A_f)$ of points
of $G$ with coefficients in the ring $\A _f$ of finite adeles.

\begin{notation} If  $a,b \in  \G$ are elements  in an  abstract group
$\G$,  we write  $^a(b)=aba^{-1},[a,b]=aba^{-1}b^{-1}$.  If  $m\in \G$
and  $A\subset \G$  is  a  subset, we  denote  by  $^m(A)$ the  subset
$mAm^{-1}$.  If   $A,B$  are  subgroups,  then   $[A,B]$  denotes  the
``commutator     subgroup''    generated     by    the     commutators
$[a,b]=aba^{-1}b^{-1}$  for  $a\in  A,   b\in  B$.   Abusing  notation
slightly, we  sometimes write, for  subsets $X,Y \subset  \G$, $[X,Y]$
for the {\it set} of commutators $[x,y]$ with $x \in X, y\in Y$. If $M
\subset \G$ is a subset and $A\subset  \G$ is a subgroup, we denote by
$^M(A)$ the {\it subgroup} generated by the conjugates $^m(A)$.

\end{notation}

\subsection{Maps in the  Congruence topology}

Suppose ${\mathbb A}^a,{\mathbb A}^b$  are affine spaces of dimensions
$a,b$   respectively. We  fix   the  standard   bases   of   ${\mathbb
A}^a,{\mathbb  A}^b$; then  the  coordinate  functions $X_i,Y_j$  with
respect to  these bases  generate the  respective coordinate  rings of
${\mathbb A}^a, {\mathbb A}^b$.   The polynomial rings $\Z[X_i]_{1\leq
i \leq  a}$ and $\Z[Y_j]_{1\leq  j \leq b}$ with  integer coefficients
will be referred to as  rings of integral polynomials. Suppose $\theta
: {\mathbb A}^a \ra {\mathbb A}^b$  is a morphism of varieties defined
over  $\Q$. Then  it follows that there exist  finitely many
vectors $w_{\nu} \in \Q ^b$, indexed  by a finite set of elements $\nu
\in  {\mathbb Z}_+^a$  (thus  the $\nu  $  have non-negative  integral
entries) such that
\[  \theta  (X)= \sum_{\nu}  w_{\nu}  X^{\nu}.\]  By taking  a  common
denominator $D$  of the  finite set $w_{\nu}$  of vectors,  it follows
that
\[\theta  (X)= \frac{1}{D}\sum_{\nu }   w'_{\nu}  X^{\nu},\]   for  some
integral vectors $w'_{\nu}$. We get from the preceding that
\begin{equation} \label{integraltaylor}
  \theta (X)=\theta  (0)+ \frac{1}{D} \sum _{\nu \neq  0} w''_{\nu}X^{\nu}=
  \theta (0) + \frac{1}{D} \sum _{i=1}^a X_iP_i(X) ,
\end{equation} for  some {\it integral} vectors  $w''_{\nu}$, and some
${\mathbb A} ^b$-valued  integral  polynomials  $P_i$ on  the  affine  space
${\mathbb A}^a$.

If $a \geq 1$ is an integer, we can write, from (\ref{integraltaylor}), that
\begin{equation} \label{denominatortaylor}  \theta \big{(}\frac{1}{a}X
\big{)}=\theta (0)+ \frac{1}{Da^N} \sum _{\nu \neq 0} v_{\nu}X^{\nu}
\end{equation} for  some {\it integral}  vectors $v_{\nu} \in  \Z ^b$,
and  for  some integer  $N$  depending  on  the  total degree  of  the
polynomial $\theta $.

\begin{lemma} \label{Dlemma} Suppose $\theta :  H \ra G$ is a morphism
of algebraic groups defined over $\Q$. Then there exists an integer $D
\geq  1$ such  that  for all  $m \geq  1$,  $\theta (H(Dm\Z))  \subset
G(m\Z)$.
\end{lemma}

\begin{proof} Fix  an embedding  $G \subset  SL_B$ defined  over $\Q$;
then  $G$ is  a $\Q$-defined  closed sub-variety  of $M_B$,  the affine
space  of dimension  $B^2$  viewed  as the  vector  space of  $B\times
B$-matrices (with  the matrices $E_{ij}$  -whose entries are  all zero
except the  $ij$-th entry  which is $1$-  forming a  preferred basis).
Similarly  $H \subset  SL_A \subset  M_A$ as  a $\Q$  sub-variety.  The
composite  $\theta: H  \ra  G \subset  M_B$ is  then  a matrix  valued
polynomial function on $H$. By the  definition of the topology on $H$,
$\theta  $ is  then a  polynomial  function (also  denoted $\theta  :X
\mapsto \theta (X)$)  on $M_A$ with values in $M_B$,  and defined over
$\Q$ i.e.  has coefficients in $\Q$ with respect to the basis $X_{ij}$
of matrix entries. Consequently  $\theta (h)=\frac{1}{D} N(h)$ for all
$h\in M_A$ where the ``denominator'' $D$  is an integer so chosen that
it is  divisible by the  denominators of  all the coefficients  of the
polynomial $\theta$, and  the ``numerator'' $N(h)$ is  a polynomial on
$M_A$ with {\it integer} coefficients.  \\

In particular, if $X,Y \in M_A(\Z)$ are matrices with integer entries,
then  (cf. equation  (\ref{integraltaylor})) we  get $\theta  (X+Dm Y)
=\frac{1}{D} N(X+DmY)$ and  $N(X+DmY)-N(X)=DmP(X,Y)$ where $P(X,Y)$ is
a  $M_B$  valued polynomial  function  on  $M_A\times M_A$  with  {\it
integer} coefficients. Hence
\[  \theta   (X+DmY)-\theta  (X)=mP(X,Y)\equiv   0~(mod  ~   m).\]  In
particular, taking $X=Id_A \in M_A(\Z)$ we get $\theta (H(Dm\Z))\subset
Id _B+mM_B(\Z)$ is an integral matrix congruent to identity modulo $m$
and this proves the lemma.  \\

(The same  construction (see equation  (\ref{denominatortaylor})) shows
that if  $N$ is  the total  degree of the  polynomial $\theta  (h)$ as
above, then 
\begin{equation} \label{denominatorbound}  
\theta (\frac{1}{a}M_A(\Z)) \subset \frac{1}{Da^N}M_B(\Z)).
\end{equation} 
\end{proof}

\begin{notation} Let $i: M \ra G$  be an embedding of linear algebraic
groups  defined  over  $\Q$;  the  groups  $M,G$  are  Zariski  closed
$\Q$-defined  sub-varieties of  $M_B$ say,  where $M_B$  is the  affine
space of dimension $B^2$ viewed as  the set of $B\times B$ matrices. A
polynomial on $M_B$ is said to be {\bf integral} if it is a polynomial
in  the  matrix  entries  of $M_B$  with  integral  coefficients.   An
integral polynomial from $M_B \times  M_B$ into $M_B$ may similarly be
defined. \\

Consider the commutator map $M \times  G \ra G$ given by $(t,x)\mapsto
txt^{-1}x^{-1}$. Since $G \subset SL_B$, the  image of $x\in G$ is the
{\it  adjoint matrix},  denoted $A(x)$,  of $x$.  Clearly, the  map $x
\mapsto A(x)$  extends to an  integral polynomial on $M_B$.  Hence the
commutator map  $(t,x)\mapsto txt^{-1}x^{-1}$ extends to  the integral
polynomial   $M_B  \times   M_B  \ra   M_B$  given   by  $(t,x)\mapsto
txA(t)A(x)$. \\

If $r\in  \Q$ is a nonzero  rational number, denote by  $rM_B(\Z)$ the
set  of  matrices  whose  entries  are  $rX_{ij}$  with  $X=(X_{ij}) \in
M_B(\Z)$.
\end{notation}

\begin{lemma}  \label{commutatorlemma}  With the  preceding  notation,
there exists an  integer $e\geq 1$ such that we  have the inclusion of
the commutator
\[ [M(a ^e \Z),  G\cap (1_B+\frac{m}{a}M_B(\Z))] \subset G(m\Z),\] for
all integers  $m\geq 1$ and $a  \geq 1$ (i.e. the  integer $e$ depends
only on the embedding $M\subset G$ and not on $m$ and $a$).
\end{lemma}

\begin{proof} If $x \in  G\cap (1_B+\frac{m}{a}M_B(\Z))$, then $x=1+X$
with $X  \in \frac{m}{a}M_B(\Z)$. The  formula for the  adjoint matrix
$A(x)$ of $x$ shows that if  we write $A(x)=1+X'$, then the entries of
the matrix $X'$  are {\it integral} polynomials in the  entries of the
matrix   $X$    of   degree   at   most    $B^2$.   Consequently,   we
have   \begin{equation}   \label{adjoint}   A(x)=1+X',   \quad   X'\in
\frac{m}{a^{B^2}}M_B(\Z).\end{equation}

If $t\in M(a^e\Z)$ and $x \in G\cap (1+\frac{m}{a}M_B(\Z))$, we see that 
\[ txt^{-1}x^{-1}=t(1+X)t^{-1}(1+X')= (1+tXt^{-1})(1+X')=\]
\[= 1+X+X'+XX'+ (tXt^{-1}-X)+(tXt^{-1}-X)X'. \]
Since $xA(x)=1$ for $x \in G$, we get $1+X+X'+XX'=1$. Hence we get  from the preceding equation, 
\[txt^{-1}x^{-1}=1+[t,X]t^{-1}+[t,X]t^{-1}X',\] 
where $[t,X]=tX-Xt$ is the matrix commutator. Write $t=1+T$. Since $t\in M(a^e\Z)$, it follows that $T\in a^eM_B(\Z)$. We get 
\begin{equation} \label{commutatorequation}
txt^{-1}x^{-1}=1+[T,X]t^{-1}+[T,X]t^{-1}X'. \end{equation}
Since $ T\in a^eM_B(\Z)$, $X \in \frac{m}{a}M_B(\Z)$  and
from equation (\ref{adjoint}),  $X' \in \frac{m}{a^{B^2}}M_B(\Z)$, it follows
that 
\[txt^{-1}x^{-1}\in 1+  ma^{e-B^2-1}M_B(\Z).\] Therefore,  if $e>B^2$,
then  $txt^{-1}x^{-1} \in  G(m\Z)$  proving the  lemma  (thus, in  the
lemma, we may take $e=1+B^2$).

\end{proof}

\section{Isotropic Groups}

\begin{notation}  $G$ is  a  $\Q$ simple  connected simply  connected,
semi-simple algebraic group defined over  $\Q$ with Lie algebra $\fg$.
Assume that $G$ is {\it isotropic}  i.e.  $\Q-rank (G)\geq 1$ and that
$P_0$ is a  minimal parabolic subgroup of $G$ defined  over $\Q$, with
$U_0=U_0^+$ the unipotent radical of  $P_0$.  Fix a Levi decomposition
(over  $\Q$) $P_0=L_0U_0$  of  $P_0$. Write  $\fp _0,{\mathfrak  u}_0,
{\mathfrak   l}_0$    for   the   Lie   algebras    of   $P_0,U_0,L_0$
respectively. The  {\it opposite} group  $U_0^{-}$ may be  defined (as
the unipotent subgroup whose Lie  algebra ${\mathfrak u}_0^{-}$ is the
subspace of $\mathfrak g$ , which,  as a module over the Levi subgroup
$L_0$, is  the dual of  $\mathfrak u _0$).   It is known  \cite {Tits}
that  the group  $G(\Q)^+$  generated by  $U_0^{\pm  }(\Q)$ is  simple
modulo the  centre of  $G(\Q)^+$.  The  resolution of  the Kneser-Tits
conjecture  (see \cite{Gille}  for the  most general  case) says  that
$G(\Q)=  G(\Q)^+$ hence  $G(\Q)$ is  also  a simple  group modulo  its
centre.  \\

The inclusion of $U_0 ^{\pm}$ in  $G$ is defined over $\Q$ and induces
a map of arithmetic completions $\widehat{U_0 ^{\pm}} \ra \widehat{G}$
of the groups $U_0 ^{\pm}(\Q)$ and $G(\Q)$ respectively; since ( as is
well  known)  the  unipotent   groups  $U_0  ^{\pm}(\Q)$  satisfy  the
congruence   subgroup   property,   it  follows   that   $\widehat{U_0
^{\pm}}=U_0 ^{\pm }(\A_f)$: thus the exact sequence
\[1 \ra  C \ra \widehat{G} \ra  \overline{G} \ra 1,\] splits  over the
subgroups $U_0 ^{\pm}(\A _f)$. 
\end{notation}

\begin{notation} ({\bf Definition of $M$})  

Let $P$ be a proper {\it maximal} parabolic subgroup of $G$ containing
$P_0$, with unipotent radical $U$ with  $U\subset U_0$, and fix a Levi
decomposition $P=LU$  of $P$ with  $L \supset  L_0$.  We write,  as we
may,  $L=S_1S_2 M_sM'$  as  an  almost direct  product  of: $S=S_1$  a
maximal  $\Q$-split  torus,  $S_2$  a maximal  torus  which  is  $\Q$-
anisotropic and  $M_s$ (if not the  trivial group ) is  the product of
all semi-simple  and $\Q$ simple  {\it isotropic groups }  (i.e.  those
which have  a $\Q$-split torus  contained in  it) factors of  $L$, and
$M'$ the  product of  semi-simple $\Q$  simple anisotropic  factors of
$L$.   Since $P$  is a  maximal  proper parabolic  subgroup, we  have:
$S= S_1\simeq {\mathbb G}_m$ has $\Q$-rank one. \\

The $\Q$-rank of  $L$ is then $1$ (the dimension  of $S_1$) plus the
$\Q$ rank of $M_s$; the rest  of the factors have $\Q$-rank zero. \\

[1] If $\Q-rank (G)\geq 2$, the preceding paragraph implies that $M_s$
has  positive dimension.   In particular,  the connected  component of
identity  of the  Zariski closure  of  (the group  of integer  points)
$L(\Z) $ contains $M_s$; in this case we write $M=M_s$. \\

We  note that  if  $\Q$-rank of  $G$  is one,  then  $P=P_0$, $S_1$  has
dimension  one,  and $M_s$  is  trivial.   Consequently, $L_0(\Z)$  is
contained  in $S_2M'$  ( up  to finite  index). We  write $M$  for the
connected component of identity of  $L_0(\Z)$. We will distinguish the
cases \\

[2] $\Q -rank (G)=1$ and $M$ is abelian and \\

[3] $\Q -rank (G)=1$ and $M$ is not abelian. \\

In all  these cases, since  $G$ is  simply connected, it  follows that
$M_s$ is also  simply connected.  The simplicity - modulo  centre - of
$G^+(\Q)$  implies  (  as  is easily  seen)  that  $U^{\pm}(\Q)$  also
generates $G^+(\Q)$.
\end{notation}

Since  the unipotent  groups  $U^{\pm}$ have  the congruence  subgroup
property, every arithmetic subgroup in $G(\Z)$ contains $U^{\pm}(m\Z)$
for  some integer  $m$. Thus  the  inclusion $U^{\pm}  \ra G$  induces
embeddings   $U^{\pm}(\A   _f)    \ra   \widehat{G}$   of   arithmetic
completions. Denote  by $E(m)$  the {\it  normal} subgroup  of $G(\Z)$
generated by $U^{\pm}(m\Z)$. It may or  may not have finite index, but
it is easily  seen that the closure  of $E(m)$ in $G(\A  _f)$ is open.
Thus, there  is a smallest  congruence subgroup containing  $E(m)$; we
call   it   the  congruence   closure   of   $E(m)$  and   denote   it
$Cl(E(m))$. Moreover, $C$  is the inverse limit as $m$  varies, of the
profinite completions of the quotient $Cl(E(m)/E(m)$:
\begin{equation}   \label{elementaryinverselimit}   C  =   \varprojlim
\widehat { (\frac{Cl(E(m))}{E(m)}) },
\end{equation} where the roof denotes  the profinite completion of the
group  involved, and  $m$  varies  over all  integers  as in  equation
(\ref{Casinverselimit}). \\

Denote by $C'$ the image of the congruence kernel $C_L$ in $C$ induced
by the inclusion $L \ra G$.
\begin{lemma}  \label{C'central}  The group  $C'$  is  central in  the
arithmetic completion $\widehat{G}$ of $G(\Q)$.
\end{lemma}

\begin{proof}  The  group  $\widehat  {L}$  acts  on  the  completions
$\widehat {U^{\pm}}  \simeq U^{\pm  }({\mathbb A}_f)$.  However, being
linear, this action descends to an action of the congruence completion
$\overline{L}$  of  $L(\Q)$.  Consequently,   the  kernel  $C_L$  acts
trivially on $U^{\pm}(\A  _f)$. Since $U^{\pm }(\Q)$ is  a subgroup of
$U^{\pm}(\A _f)$ it follows that  $C'$ commutes with $U^{\pm}(\Q)$ and
hence with the group generated by  them. It was already seen that this
subgroup is $G(\Q)$  (using the resolution of  the Kneser-Tits problem
\cite{Gille}) and hence  is dense in $\widehat G$.  Therefore, $C'$ is
central in $\widehat G$.
\end{proof}

The group $C/C'$  is the kernel to the map  $\widehat{G}/C'$.  For any
integer $k$ denote by $F(k)$  the normal subgroup in $G(\Z)$ generated
by  $P(k \Z),  U^-(k  \Z)$.  Then  $F(k) \subset  \G  (k)$ and  $F(m)
\subset  \G  (m)$. Let  $Cl  (F(m))$  denote the  smallest  congruence
subgroup of $G(\Z)$ containing $F(m)$. \\

The  congruence kernel  $C$  is  the inverse  limit  of the  profinite
completion $\widehat{\G(m)/E(m)}$ as $m$ varies; the action of $M(\Z)$
on all these  groups induces an action on $C$  and these actions
are compatible. By Lemma  \ref{C'central}, the group $C'$ is central
and   $C/C'$  is   the   kernel  to   the   map  $\widehat{G}/C'   \ra
G(\A_f)$. Moreover, the  topology on $P(\Q)$ induced  by its inclusion
in $G(\Q) \subset \widehat{G}/C'$ is  the congruence topology since we
have gone modulo the congruence kernel  of $L$ = the congruence kernel
of    $P$.      We    then     have,    analogously     to    equation
(\ref{elementaryinverselimit}),

\begin{equation} \label{parabolicinverselimit}
C/C'= \varprojlim \widehat{(\frac{Cl(F(m))}{F(m)})}   
\end{equation}
expressing $C/C'$ as an inverse  limit of the profinite completions of
the quotient groups $Cl(F(m))/F(m)$.

\subsection{The open set ${\mathcal U} =U^{-}P \subset G$}

The map $U^-\times  P \ra G$ given by  multiplication $(u^-,p) \mapsto
u^-p$ is a  map of affine $\Q$-varieties which is  an isomorphism onto
its image which is  a Zariski open set $\mathcal U$ in  $G$. If $g \in
{\mathcal U}  (\Q)=u^{-}p$, then the uniqueness  of this decomposition
says that $u^{-}\in U^-(\Q), \quad p\in P(\Q)$. \\

Given a  prime $p$,  the subset  $U^-(\Z _p)P((\Z  _p)={\mathcal U}(\Z
_p)$ is open in the $\Q _p$  topology and contains $1$. Hence there is
a compact open subgroup $K_p$ of  $G(\Z _p)$ contained in $U^-( \Z _p)
P(\Z _p)$. \\

The  conjugation  action of  $M$  on  $G$  stabilises all  the  groups
$U^{\pm},M ,P^+$ and  hence $M$ stabilises the open  set $\mathcal U$.
Fix an integer $m$;  if $S$ is the set of  primes dividing $m$, denote
by  $R$ the  ring $\prod  _{p\in S}  \Z _p$;  this is  a compact  open
sub-ring in $\Q_S=\prod  _{p\in S} \Q _p$ and the  principal ideal $mR$
generated  by $m$  is a  compact open  ideal in  $R$. Thus  ${\mathcal
U}(mR)=U^-(mR)P(mR)$  is  an open  set  (containing  identity) in  the
product  group $G_S=\prod  _{p\in S}  G(\Q _p)$.   By the  topology on
$G_S$  there  exists  a  compact  open  subgroup  $K_S=G(m'R)  \subset
{\mathcal U}(mR)$ of  $G_S$ for some integer $m'$. We  may choose $m'$
to be divisible only by primes in $S$ since the other primes are units
in  $R$.  Hence  $K_S$  is  open normal  of  finite  index in  $G(R)$.
Consequently  $\Gamma (m')  =G(\Z) \cap  K_S$ is  a congruence  normal
subgroup in $G(\Z)$ with $m'$ divisible by $m$ and only by primes that
divide $m$. \\

We now compute the action of $M(\Z)$ on the latter group $\G(m')$. \\

Given $x\in \G(m')$, write, as  we may, $x=u^-p$ with $u^-\in U^-(\Q)$
and $p \in P(\Q)$.  On the other  hand, viewed as an element of $G(R)$
we have $x=u^-_Rp_R$ with $u^-(mR),  p_R \in P(mR)$. The uniqueness of
decomposing  an element  of $G(\Q)$ as a product $u^-p$  then shows that
$u^-=u^-_R$, $p=p_R$  and are  therefore integral in  $\Z _p$  for all
primes $p$ dividing $m$ and hence have denominators co-prime to $m$ and
all off diagonal entries have numerators divisible by $m$. \\

We have  fixed a linear  $\Q$-embedding $G \subset SL_B  \subset M_B$;
Thus $u^-,p$ and $p^{-1}$ viewed as  matrices in $M_B(\Q)$, are  of the form
identity plus  a matrix having a  common denominator $Da^N$ (  we write
$a^N$ instead of $a$ keeping in mind a future application) with $a$ is
co-prime to $m$ and numerators which are all divisible by $m$:
\[u^-\in U^-\cap (1+\frac{m}{Da^N}M_B(\Z)), \quad  p, p ^{-1} \in P\cap
(1+\frac{m}{Da^N}M_B(\Z)).\]

By Lemma  \ref{commutatorlemma}, there  exists an  integer $e  \geq 1$
($e$ independent of  the $a$ chosen, and depend only  on the embedding
$G \subset SL_B$; the  $a$ in the lemma is replaced  by $Da^N$) so that
we have the inclusion of the commutator subgroups:
\[ [M(D^ea^{eN}), U^- \cap  (1+\frac{m}{Da^N}M_B(\Z)] \subset U^-(m\Z) \subset
F(m), \] and
\[ [ M(D^ea^{eN}),  P\cap  (1+\frac{m}{Da^N}M_B(\Z)]  \subset  P(m\Z)  \subset
F(m).\]

Fix  $x\in G(m')$;  this determines  the  integer $a=a(x)$  is in  the
preceding  paragraph. Fix  $t  \in  M(D^ea^{eN}\Z)$.  We compute  the
conjugate $^t(x)$ for $x\in G(m')$:  write $x=u^-p$; we have seen that
$u^-,p$ have denominators $Da^N$ co-prime  to $m$ and numerators divisible
by $m$.  Then
\[^t(x)=^t(u^-)             ^t             (p)=             tu^-t^{-1}
  (u^-)^{-1}u^-pp^{-1}tpt^{-1}=
\]
\[=  [t,u^-]u^-p  [p^{-1},t]=[t,u^-]x[p^{-1},t]\]  and  the  foregoing
inclusion of commutator subgroups shows that
\begin{equation} \label{conjugationbyM} ^t (x)\in F(m)xF(m)=xF(m),
\end{equation}  (the  last equality  holds  since  $F(m)$ is  normal).
Therefore, the congruence group $M(a^{eN}\Z)$ fixes the coset $xF(m)\in
G(m')F(m)/F(m)$ through the element  $x$; the integer $a=a(x)$ depends
on $x$ and is co-prime to the integer $m$.

\subsection{Centrality in the semi-local case}

If $\Z$ is  replaced by the semi-local subring $\Z_X$  (i.e. $X$ is the
complement of  a {\it  finite set}  $S$ of primes of  $\Q$), then  for the
group $G(\Z_X)$  the congruence  subgroup kernel is  central: consider
the completion $\widehat{G}_A$  of the group $G(\Q)$,  with respect to
the profinite  topology on  the group $G(A)$.   We have  the analogous
exact sequence ($A=\Z_X$ in the following paragraph)
\[1 \ra D \ra {\widehat G}_A \ra G(\widehat A) \ra 1.\] 

The ideals in $A$ are of the  form $mA$ where $m$ is divisible only by
primes which are in the complement $S$ of  $X$ (the complement of $X$ is a
finite set by assumption).

The  group $D$  is  the inverse  limit (as  $m$  varies over  integers
divisible   only   by   primes   in $S$)   of
$\widehat{G(mA)/E(mA)}$   where  the   roof   denotes  the   profinite
completion (actually, the group is  finite in the semi-local case, but
we do not need  to use it).  Every element of  a finite index subgroup
$G(m'A)E(mA)/E(mA)$ ($m'$ as before) may  be replaced by an element of
the form $u^-p=u^-zu$,  where $u^-\in U^-, z\in L, u  \in U$.  But the
elements $u^-$  and $u$ already  lie in $U^{\pm} \cap  G(mA_X) \subset
E(m)$ since the  denominators of the matrix entries  of these elements
are co-prime  to $m$.   It follows  that $D$ is  the image  $C_L'$ (the
congruence subgroup  kernel of  $L$) and is  hence centralised  by the
central  torus $S(\Q)$  in  $L(\Q)$.  However,  all  of $G(\Q)$  still
operates  on  $D$  but  $S(\Q)$  acts  trivially;  therefore,  by  the
simplicity  of  $G(\Q)$  modulo  the   centre,  all  of  $G(\Q)$  acts
trivially; That is, the exact sequence
\[ 1 \ra D  \ra {\widehat G(A)} \ra G({\widehat A}) \ra 1   \]
has central kernel $D$. 

\begin{remark}  The  group $D$  is  actually  trivial; the  congruence
subgroup property  for general  $G$ in the  semi-local case  has almost
been  proved   (\cite{Sury},  \cite{Pr-R2})  but  we   only  need  the
centrality here.
\end{remark}  

\subsection{Commuting subgroups}

The following  proposition was  observed in \cite{R1}  and is  in fact
used in several proofs of centrality (see \cite{Pr-R2}).

\begin{proposition} \label{commutinggroups} Denote, for each prime $p$
of $\Q$,  by $G_p$ the  closed subgroup  of $\widehat G$  generated by
$U^{\pm}(\Q_p)$. Then  $C$ in central in  $\widehat G$ if and  only if
for each pair $p\neq q$, the groups $G_p$ and $G_q$ commute.
\end{proposition}

\begin{proof} Suppose $G_p,G_q$ commute for different primes $p,q$. We
have the exact sequence
\[1 \ra C_p =C \cap G_p \ra G_p \ra G(K_p) \ra 1.\]
This yields the exact sequence 
\[ 1 \ra C^p  \ra G^p \ra G(\A_{K_f\setminus \{ p\}})  \ra 1 .\] Here,
$\A_{f\setminus \{p\}}$  denotes the  sub-ring of  the ring  of finite
adeles of  $\Q$ which  is a  restricted direct  product of  $\Q_l$ for
primes different  from the  prime $p$.  $G^p$  is the  closed subgroup
generated by $G_q$  with $q\neq p$.  $C^p$ is the  intersection of $C$
with the closed subgroup $G^p$. \\

The group $G_p$ is normal in $\widehat G$ since it is normalised (even
centralised) by $U^{\pm}(\Q_q)$  for each $q\neq p$  and normalised by
(indeed, contains)  $U^{\pm}(\Q _p)$ by assumption.   Therefore, $G_p$
is normalised by $U^{\pm}(\Q)$. The  group generated by $U^{\pm }(\Q)$
is, by  the solution to the  Kneser-Tits problem, all of  $G(\Q)$, and
since  $G(\Q)$ is  dense in  $\widehat G$,  it follows  that $G_p$  is
normalised by $\widehat G$. Hence $G_p$ is a closed normal subgroup of
$\widehat G$.  Therefore, so is $G^p$.  Thus we may form  the quotient
${\widehat G}/G^p$  which is a  quotient of  $G_p$. We have  the short
exact sequence
\[  1 \ra  C/C^p \ra  {\widehat G}/G^p  \ra G(\Q_p)=G(\A  _f)/G(\A _{f
\setminus \{p\}})  \ra 1 \]  where in the quotient  ${\widehat G}/G_p$
the closure of the group $G(A_p)$ is a profinite group and maps to the
congruence  completion $G(\Z_p)$;  $A_p$ here  is the  semi-local ring
consisting  of rational  numbers of  the form  $\frac{a}{b}$ with  $b$
co-prime to the prime $p$. Moreover, we have the exact sequence 
\[ 1 \ra C/C^p \ra \widehat{G(A_p)} \ra G(\Z _p) \ra 1.\]

By the  preceding subsection,  the group  $G(A_p)$ has  the congruence
subgroup  property  (in  the  sense  that  the  associated  congruence
subgroup kernel is central). Hence the extension $C/C^p$ is central in
$\widehat G/  C^p$ and hence the  commutator subgroup $C'=[C,{\widehat
G}]$ is contained  in $C^p$ for every prime $p$.   In particular, $C'$
is  centralised  by $G_p$  for  every  prime  $p$  and hence  $C'$  is
centralised by $\widehat G$. \\

Hence for $g\in {\widehat G}$ and $c\in C$, the map $\psi : g\mapsto
gcg^{-1}c^{-1}$ is  a homomorphism into the central  subgroup $C'$. In
view of  the simplicity of  $G(\Q)$ this means  that the map  $\psi$ is
trivial and hence that $C$ is central in $\widehat G$. This proves the 
``if'' part of the proposition. \\

To  prove the  only if  part,  we argue  as follows.   Suppose $C$  is
central.  Consider an element $c$  in the commutator set: $c=[u,u^{-}]
\in [U^+(\Q  _p), U^{-}(\Q _q)]$. On  this the group $S(\Q)$  (of $\Q$
rational  points of  the split  torus $S$)  acts by  conjugation.  The
action   of  $S(\Q)$   on  $c$   is  trivial   by  assumption;   hence
$[u,u^-]=[^s(u),^s(u^-)]$   for   all   $s   \in   S(\Q)$.    By   weak
approximation,  $S(\Q)$   is  dense  in  the   product  $S(\Q_p)\times
S(\Q_q)$. It follows by the density that
\[[u,u^{-}]= [^{s}(u), u^{-}] \] for all $s\in S(\Q _p)$. Since we can
choose a sequence  $s_k \in S(\Q _p)$ such  that $^{s_k}(u)$ contracts
to identity  we get  $[u,u^-]=1$. Since $u$  is arbitrary,  it follows
that $U^+(\Q_p)$ commutes with $u^{-}$. In other words, $G_p$ commutes
with $G_q$.
\end{proof}

\section{An Application of Dirichlet's Theorem}

\subsection{Dirichlet theorem on primes}

Let $M\subset GL_n$ be a linear algebraic group defined over $\Q$. Fix
a prime $l$, and an integer $m \geq 2n$. The unit group $(\Z/l^m\Z)^*$
is cyclic of order $l^{m-1}(l-1)$ if $l$  is odd and if $l=2$, then it
has an element of order $2^{m-2}$.  Consider the set $S$ of primes $p$
such that the order of $p$  modulo $l$ is either $l^{m-1}(l-1)$ if $l$
is odd and  $2^{m-2}$ if $l=2$.  By Dirichlet's  theorem on infinitude
of primes in arithmetic progressions, the set $S$ is infinite.

Fix  $l$ and  $p \in  S$ as  above. Write  $e=e_l(p)$ for  the largest
integer such that  the finite group $M({\mathbb F}_p)$  has an element
of exponent $l^e$.

\begin{lemma} The exponent $e_l(p)$ satisfies the estimate 
\[e_l(p)\leq [\frac{n}{l-1}]+ [\frac{n}{l(l-1)}]+ \cdots+ [\frac{n}
{l^{m-1}(l-1)}] ,\] where $[x]$ denotes the integral part. 
\end{lemma}

\begin{proof} We have  $M \subset GL_n$. We need  only prove the lemma
for  $GL_n$ since  the  exponent of  the  subgroup $M({\mathbb  F}_p)$
divides the  exponent of the  larger group $GL_n({\mathbb  F}_p)$. The
order  of   $GL_n({\mathbb  F}_p)$  is
\[(p^n-1)(p^n-p^2)\cdots  (p^n
-p^{n-1}).\]

Now  $p$   is  co-prime  to   $l$,  and  generates  the   cyclic  group
$(\Z/l^m\Z)^*$ (if  $l$ is odd). For  $j\leq m $, let  $X(l^j)$ denote
the  set of  $i\leq n$  with $p^i\equiv  1 (~mod  \quad l^j)$  and let
$x(l^j)$ be  the cardinality of  the set $X(l^j)$.  The  assumption on
$p$ implies  that if  $i \in  X(l^j)$, then $i$  must be  divisible by
$l^{j-1}(l-1)$.   Hence the  number  of  $i \leq  n$  such that  $l^j$
divides        $p^i-1$        is       the        integral        part
$[\frac{n}{l^{j-1}(l-1)}]=x(l^j)$.\\

The two  preceding paragraphs  imply that  for each  $j \leq  m$, then
number of factors in the product $\prod _{i=1}^n (p^i-1)$ divisible by
$l^j$ is  $x(l^j)$ (=  the integral  part $[\frac{n}{l^{j-1}(l-1)}]$).
Now,  $X(l)\supset  X(l^2)\supset  \cdots \supset  X(l^m)$;  if  $i\in
X(l^j)\setminus X(l^{j+1})$,  then the  highest power of  $l$ dividing
the factor $p^i-1$  is $j$. Therefore, the largest power  of $l$ which
divides     the     order      of     $GL_n({\mathbb     F}_p)$     is
$1(x(l)-x(l^2))+2(x(l^2)-x(l^3))                              +\cdots+
mx(l^m)=x(l)+x(l^2)+x(l^3)+\cdots +x(l^m)$. \\

The  last two  paragraphs  imply the  lemma  for $GL_n$  and hence  for
arbitrary $M\subset GL_n$.
\end{proof}

\begin{corollary}   \label{2ncorollary}    If   $l\geq    n+2$,   then
  $e_l(p)=0$. Moreover, in all cases, $e_l(p)\leq 2n$. Let $R_l=l^{e_l(p)}$.
  Then $R_l \leq (n+1)^{2n}$. 
\end{corollary}

\begin{proof} The formula  for $e_l(p)$ as a sum in  the lemma is such
that   all    the   terms    are   bounded    by   the    first   term
$[\frac{n}{l-1}]$. This  first term  is $0$  if $  l \geq  n+2$. Hence
$e_l(p)=0$.

The formula also shows that
\[ e_l(p)\leq \frac{n}{l-1}+\frac{n}{(l-1)l}+ \cdots \leq 2n.\] Hence 
for $l\geq n+2$, $R_l=1$ and otherwise, $R_l=l^{e_l(p)} \leq (n+1)^{2n}$. 

\end{proof}

As before, we let $M \subset  GL_n$ an algebraic subgroup defined over
$\Q$ and $l$  a prime.  Let $a  \geq 1,b \geq 2$  be co-prime integers.
Consider the arithmetic progression $a+bx: x=0,1,2,\cdots $. For $m\in
\Z$, denote by  $M(m)$ the principal congruence subgroup  of level $m$
in $M(\Z)$, and $R_l(m)$ be the $l$-exponent (the largest power of $l$
which divides the exponent) of the finite group $M(\Z)/M(m\Z)$. Denote
by  $R_l(a,b)$  the  infimum of  $R_l(a+bx):  x=0,1,2,3,\cdots$.   The
$l$-exponent  of  the quotient  $M(\Z)/\D$  where  $\D$ is  the  group
generated  by $\{M(a+bx);x\in  \Z\}$  is clearly  no  bigger than  the
$l$-exponent   $R_l(a+bx)$   of   $M(\Z)/   M((a+bx)\Z))$   for   each
$x$. Therefore the $l$ exponent of $M(\Z)/\D$ is $\leq R_l(a,b)$.

\begin{proposition} There  is an integer  $R_l$ depending only  on $n$
and not on $a,b$ such that $R_l(a,b) \leq R_l$.
\end{proposition}

\begin{proof}  Let $l^e$  be the  largest power  of $l$  dividing $b$;
hence $b=l^ec$  with $c$ co-prime to  $l$.  Fix an integer  $m \geq 2n$
{\it and} $m\geq e$.  We find a  prime $p$ such that $p$ generates the
unit   group  $(\Z/l^m\Z)^*$.    Thus  $p\in   S$  of   the  preceding
lemma. Suppose $a=p^h$  modulo $l^m$.  By Dirichlet's  theorem, we can
find  arbitrarily large  distinct primes  $p_1, \cdots,  p_{h-1}$ such
that  $p_j \equiv  1 \quad  (mod ~c)$,  and $p_j  \equiv p  \quad (mod
~l^m)$ ($j\leq  h-1$).  We  also choose a  large prime  $p_h$ distinct
from $p_1,  p_2, \cdots,  p_{h-1}$ such that  $p_h\equiv a  \quad (mod
~c)$  and $p_h\equiv  p  \quad (mod  ~l^m)$.   Then $p_1p_2\cdots  p_h
\equiv a  \quad (mod ~b)$.   Now $GL_n(\Z/p_1p_2\cdots p_h\Z)$  is the
product  of $GL_n({\mathbb  F}_{p_i})$, and  the $l$  exponent of  the
product is the supremum of the exponents of $GL_n({\mathbb F}_{p_i})$.
The  latter  is, by  the  lemma,  bounded  only  by an  integer  $R_l$
depending  on  $n$:  $R_l  \leq  l^{2n}  \leq  (n+1)^{2n}$  (Corollary
\ref{2ncorollary}).  The proposition follows.

\end{proof}

We are  now in a  position to prove Theorem  \ref{dirichletlinear}. We
restate it  ( to  save notation, we  replace $bm$ by  $b$, as  we may,
since $a$ is still co-prime to $bm$) as follows.

\begin{proposition} \label{dirichlet}  Let $M\subset GL_n$ be  a linear
algebraic group defined over $\Q$, and $a,b$ two co-prime integers. Let
$M(\Z)$ be  fixed, and denote by  $N$ the (normal) group  generated by
the congruence  subgroups $M(ax+b)$ with  $x\in \Z$ of  $M(\Z)$. There
exists an integer $R$ depending only  on $n$ such that the exponent of
every element of  $M(\Z)/N$ divides $R$; in particular,  the group $\D=
M(\Z)^R$ generated by $R$-th powers of $M(\Z)$ lies in $N$.

\end{proposition}

\begin{proof} This is a simple  application of the preceding lemma. We
first prove this for elements whose exponents are a power of the prime
$l$. In that case, the integer $R$ is the $R_l \leq (n+1)^{2n}$ of the
preceding  lemma.   Moreover,   by  Corollary  \ref{2ncorollary},  the
$l$-exponent is $1$ unless $l$ is a prime with $l\leq n+1$.  Hence the
exponent of $M/N$ is
\[  R =  \prod _{l\leq  n+1}R_l  \leq \prod  _{l \leq  n+1}(n+1)^{2n}=
\big{[}(n+1)^{2n}\big{]}  ^{\pi (n+1)},\]  where  $\pi  (n+1)$ is  the
number  of  primes $l  \leq  n+1$.  This  proves the  proposition  and
equivalently Theorem \ref{dirichletlinear}.

\end{proof}

\newpage
\section{Proof of  centrality when $\Q  -rank (G)=1$ and
  $M_0 (\Z)$ is virtually abelian}

\subsection{generalities}

\begin{notation}  For general  results on  reductive algebraic  groups
over  arbitrary fields,  we  refer to  the article  of  Borel and  Tits
\cite{BT}. \\

We assume  that $\Q -rank  (G)=1$.  Then the proper  maximal parabolic
subgroup $P=P_0$  is a minimal  parabolic subgroup defined  over $\Q$,
$U$ the unipotent  radical of $P$, and $P=LU$ a  Levi decomposition of
$P$. Take $S\subset  L$ a maximal $\Q$-split torus in  $L$ (and in $G$
since  $\Q-rank (G)=1$.   Moreover, since  $\Q-rank (G)=1$,  the group
$L(\Q)$ consists entirely of semi-simple elements. \\

By Corollary (5.8)  of \cite{BT}, the roots of $S$  on the Lie algebra
$\fg$ form a (not necessarily reduced)  root system, of rank one.  But
in a  root system,  the only  multiples of  a root  (see section  2 of
\cite{Ser2}) $\a$ are $\pm \a, \pm  2\a$. By choosing $\a$ suitably, we
see that  as a  module over  $S\simeq {\mathbb  G}_m$ the  Lie algebra
$\fu$  of  $U^+$ decomposes  as  $\fu  =\fg  _\a \oplus  \fg  _{2\a}$;
similarly the Lie algebra $\fu ^-$  of $U^-$ decomposes as $\fu ^-=\fg
_{-\a}\oplus \fg _{-2\a}$.   Since $S$ is in the centre  of $L$, these
decompositions  are stable  under the  adjoint action  of $L$.   If we
denote by $\fl$ the Lie algebra of $L$ then we have the decomposition
\[ \fg  =\fu ^-\oplus \fl \oplus  \fu = \fg _{-2  \a}\oplus \fg _{-\a}
\oplus \fl  \oplus \fg _\a \oplus  \fg _{2\a}.\] As before,  denote by
$M=ZCl (L(\Z))^0$ the  connected component of identity  of the Zariski
closure of the group $L(\Z)$ of integer points of $L$. \\

{\it  In  this section,  we  assume  that  $M$  is abelian  i.e.  that
$L (\Z)$ is virtually abelian}.
\end{notation}

\begin{lemma} $M$ is in the centre of $L$.
\end{lemma}

\begin{proof} Since $L(\Q)$ commensurates $L(\Z)$, it follows that
the identity component $M$ of the Zariski closure of $L(\Z)$ is normal
in $L$.  Now $L$ is  reductive and hence so  is $M$. Since  $M$ is
abelian,  $M$  is  a  torus   of  dimension  $d$  say.  Therefore  the
automorphism group  $M$ is  a discrete  group, namely  $GL_d(\Z)$. The
conjugation  action of  $L$  on $M$  yields  a homomorphism  $L\ra
GL_d(\Z)$; but since $L$ is connected and $GL_d(\Z)$ is discrete, it
follows that this homomorphism is  trivial. That is, $L$ centralises
$M$.
\end{proof}

\begin{lemma} \label{liealgebragenerator}

[1]  The Lie  algebra  generated by  $\fg _\a,  \fg_{-\a}$  is all  of
$\fg$. 

[2] We have $[\fg _\a,\fg _\a]= \fg _{2\a}$.

\end{lemma}

\begin{proof} We first prove that $\fg$  is generated as a Lie algebra
by the subspaces $\fg _\a, \fg _{2\a}, \fg_{-\a}, \fg _{-2\a}$. Denote
temporarily, the Lie algebra generated  by these subspaces as $\fg '$;
this is normalised by $\fl$, since  each of these spaces is. Hence all
of $\fg$ normalises $\fg'$.  The $\Q-$ simplicity of $\fg$ now ensures
that $\fg '= \fg$. \\

Under the Lie bracket, the subspace $\fg _{2\a}$ takes $\fg_{-\a}$into
$\fg  _\a$,  and $\fg  _\a$  into  $0$;  consequently, by  the  Jacobi
identity, the Lie algebra $\fg ''$  generated by $\fg _\a, \fg _{-\a}$
is normalised by the adjoint action  of $\fg _{2\a}$ (and similarly by
the action  of $\fg _{-2\a}$); hence  $\fg ''$ is normalised by $\fg
_\b$ with  $\b\in \{\pm \a, \pm  2 \a\}$; by the  preceding paragraph,
$\fg  ''$ is  normalised by  $\fg$; the  $\Q$-simplicity of  $\fg$ now
implies that $\fg ''=\fg$. This proves [1]. \\

Denote by $\fg ^*$ the direct sum subspace
\[\fg  ^*= [\fg  _\a,\fg _\a]\oplus  \fg  _\a \oplus  \fl \oplus  \fg
_{-\a}\oplus [\fg _{-\a},\fg _{-\a}] . \]

By examining  the brackets  of each  of the  individual terms  in this
direct sum,  we see  that $\fg  ^*$ is a  Lie subalgebra  defined over
$\Q$. Further  it is normalised  by the subspaces $\fg  _\a, \fg_{-a}$
(since  it  contains  them);  by  part [1],  $\fg  ^*$  is  therefore
normalised by all of $\fg$.  The $\Q$-simplicity of $\fg$ then implies
that  $\fg ^*=\fg$.  This proves  [2]: 
$[\fg _\a, \fg _\a]=\fg _{2\a}$.

\end{proof}
  
\begin{lemma}  If $t\in  M$ is  not  in the  centre of  $G$, then  the
element $t$  does not have a  non-zero fixed point in  $\fg _\a$. That
is, all the  eigenvalues of the linear transformation  $Ad (t)$ acting
on $\fg _\a$ are different from $1$. 
\end{lemma}

\begin{proof}:  Since  $M$  is  a  torus,  every  element  of  $M$  is
semi-simple.  Split the space $\fg  _\a$ into eigenspaces for $t$: $\fg
_\a =\oplus _{\l \neq 1} (\fg _\a ) _\l\oplus (\fg _\a)_1=W\oplus (\fg
_\a)_1$.  Similarly  there is  a decomposition  for $\fg  _{-\a}$.  If
possible, let  $X\in (\fg _\a)_\l$  and $Y\in (\fg _{-\a})_1$  be both
nonzero. We will show that this leads to a contradiction.  The bracket
$[X,Y]$ is  in $\fg  ^S \simeq  \fl $.  This  is impossible  since the
action  of  $t$ on  $X$  is  multiplication by  $\l$  and  on $Y$  and
$\mathfrak l$ it  is multiplication by $1$.  Therefore, $[X,Y]=0$, and
hence, by taking the sum over all  the $\l$, $[W, Y]=0$. Note that $W$
is defined over $\Q$.\\

Now take $X\in W(\Q)\setminus \{0\}$;  then $u=exp (X)$ is a unipotent
element   commuting  with   the  unipotent   element  $v=exp   (Y)\in
U^-$.  Hence $u,v$  generate a  unipotent group  $E$. However,  in the
$\Q$-rank one group $G$, every nontrivial unipotent element belongs to
a unique maximal unipotent group; hence $u\in U^+$ and hence $E\subset
U^+$; similarly, $v \in U^-$ and hence $E\subset U^-$; this means that
$E=\{1\}$ i.e. $X=0$ and $Y=0$, a contradiction. \\

This means  that either $X=0$ or  $Y=0$; in other words,  $\fg _\a$ is
either fixed  point-wise by $t$,  or $\fg  _\a=W$ has no  fixed points.
Suppose $\fg _\a =\fg _\a ^t$. By considering orthogonal decomposition
(with respect  to the Killing  form $\k$ ) of  $\fg$ as a  module over
$t\in  SO(\k) $,  one sees  that $\fg  _{-\a}=\fg _{-\a  }^t$ is  also
point-wise fixed  by $t$.  Since $\fg  $ is  generated by  $\fg _\a,\fg
_{-\a}$  by lemma  \ref{liealgebragenerator}, it  follows that  all of
$\fg$ is point-wise fixed by $t$; that is, $t$ is in the centre of $G$.
\end{proof}

\begin{corollary} Let  $V=\{exp(zX);z\in {\mathbb G}_a\subset  U^+\}$ be
the  subgroup  generated  by  the exponentials  of  an  element  $X\in
\fg_\a$. Suppose $t\in M $ is not in the centre of $G$. The commutator
subgroup $[t^\Z,V]$ contains $V$.

\end{corollary}

\begin{proof} The  Cayley -Hamilton theorem  and the fact that  $1$ is
not an eigenvalue  of $Ad (t)$ (lemma) show that  the matrix $Ad(t)-1$
is invertible  on $\fg  _\a$ and  the inverse is  a polynomial  in $Ad
(t)-1$.  Hence  $X$ is in the  span of the vectors  $(Ad (t) -1)^k(X)$
.\\

If $\mathcal G$  is a group, $R=\Z [{\mathcal G}]$  its integral group
ring and  $a,b\in \mathcal G$, then  $a-1, b-1, ab-1$ are  elements of
$R$. We have the identity
\[(a-1)(b-1) =(ab-1)  - (a-1)-(b-1) ,  \] in  $R$; this shows  that the
span of $Ad(t^k)-1$ contains the span of $(Ad t-1)^k$ for all
integers $k \geq 1$. This and the preceding paragraph show that $X$ is
a linear combination of $(Ad (t^k)-1)(X)$. \\

Since the Lie algebra of  the commutator group $[t^\Z,V]$ contains the
span  of $(Ad  (t^k)-1)(X)$  for $  k\geq 1$,  the  conclusion of  the
preceding paragraph shows that $X$ lies  in this Lie algebra; that is,
$V$ lies in $[t^\Z,V]$.
\end{proof}

\begin{notation} Let $\theta: H= SL_2 \ra G$ be a non-trivial morphism
of algebraic $\Q$-groups.  Denote by  $U_H^+$ (resp. $U^-_H$) the group
of upper  triangular (resp.   lower triangular) unipotent  matrices in
$SL_2$.   Assume $\theta  $  is such  that $\theta  (U^{\pm}_H)\subset
U^{\pm}$,  and that  the  image under  $\theta$ of  the  group $D$  of
diagonals in $H$ is the torus $S$ in $G$. 
\end{notation}

\begin{corollary}  \label{SL2commutator}  If  $\D  \subset  M$  is  an
infinite  subgroup, then  the  commutator  group $[\D,\theta  (SL_2)]$
contains $\theta (SL_2)$:
\[[\D,\theta (SL_2 )]\supset \theta(SL_2). \]

\end{corollary}

\begin{proof}  By the  preceding corollary  , the  commutator subgroup
$[\D,  \theta (U_H^+)]$  contains $\theta  (U_H^+)$; similarly,  $[\D,
\theta (U_H^-)]$ contains $\theta  (U_H^-)$. Since $SL_2$ is generated
by $U_H^{\pm}$ the corollary follows.
\end{proof}

\subsection{Some consequences of Proposition \ref{dirichlet}} 

In  this  subsection,  we  derive  some  consequences  of  Proposition
\ref{dirichlet} for $\theta  (SL_2)$ and the results  of the preceding
subsection. We consider  a nontrivial morphism $\theta: SL_2  \ra G$ f
$\Q$-algebraic groups  as before. By lemma  \ref{Dlemma}, there exists
an integer $k\geq  1$ such that $\theta  (SL_2(km))\subset G(m\Z)$ for
every  integer  $m\geq  1$.  It   follows  that,  for  every  $D\equiv
0~(mod~m)$ and every $m$, we have $\theta(SL_2(Dm)) \subset G(m\Z)$.

\begin{corollary}  \label{SL2isotropy} If  $\theta: SL_2  \ra G$  is a
nontrivial  morphism, then  there exist  integers $D,e$  such that  if
$g=\begin{pmatrix} a &  b \\ c & d \end{pmatrix}  \in SL_2(D\Z)$, then
the conjugate  $t\theta (g) t^{-1} \in  \theta (g) F(m\Z)$ for  all $t
\in  M(D(Da^N)^e\Z)$.  That  is,  $M(D(Da^N)^e\Z)$  acts trivially  on
$\theta (g)$ viewed as an element of $\Gamma (m)/F(m)$.

\end{corollary}

\begin{proof} Write
\[g= \begin{pmatrix} a & b \\ c & d \end{pmatrix}= \begin{pmatrix} 1 &
0  \\ \frac{c}{a}  & 1  \end{pmatrix}  \begin{pmatrix} a  & 0  \\ 0  &
a^{-1}  \end{pmatrix}   \begin{pmatrix}  1   &  \frac{b}{a}  \\   0  &
1\end{pmatrix} =  u^- z v\]  with $u^- \in  U_H^-$, $z$ is  a diagonal
matrix, and $v  \in U_H^+$.  Thus $u^-,  v $ are matrices  of the form
$1+\frac{km}{a}x$     with     $x     \in    M_2$.     By     equation
(\ref{denominatorbound}) there exist integers  $D,N$ such that $\theta
(u^-)\in 1_B+ \frac{m}{Da^N}M_B(\Z)$; we  may assume (by replacing $D$
by $Dk$ if necessary: cf. the last sentence of the paragraph preceding
the  corollary)) ,  that  $D$ is  divisible by  $k$.   Suppose $t  \in
M(D(Da^N)^e\Z)$; then by equation (\ref{conjugationbyM}),

\[  t\theta  (g) t^{-1} \in \theta (g) F(m),\]  preserves the
coset of $F(m)$ in $\Gamma (m)$ through the element $\theta (g)$.

\end{proof}

\begin{corollary}  \label{rankoneisotropy}  The  isotropy  of  $\theta
(g)\in  \G(m)/F(m)$  with  $g  \in SL_2(Dm\Z)$  as  in  the  preceding
corollary, contains  the group generated by the congruence groups 
$M(D (D(a+bx)^N)^e\Z)$, with $x \in Dm \Z$.
\end{corollary}

\begin{proof}  The  preceding  Corollary \ref{SL2isotropy}  says  that
$M(D(Da^N)^e)$  stabilises  the element  $\theta  (g)  \in \G(m)/F(m)$.  If
$g=\begin{pmatrix} a &  b \\ c & d \end{pmatrix}  \in SL_2(Dm)$ and $x
\in  Dm$, we  may apply  the corollary  to the  element $g'  =gu$ with
$u=\begin{pmatrix} 1 &  0 \\ x & 1 \end{pmatrix}$;  then $a(gu) =a+bx$
and  $\theta (g')  \in \theta  (g)E(m)$. Now  the preceding  Corollary
\ref{SL2isotropy}  says that  $M(D(D(a+bx)^N)^e)$  also stabilises  $\theta
(g)=\theta (g')$ in $\G(m)/F(m)$. This proves the corollary.
\end{proof}

\begin{corollary}  \label{dirichlet1} In  the notation  of Proposition
\ref{dirichlet}, denote by  $\D$ the normal subgroup  generated by the
congruence  subgroups $M(D(D(ax+b)^N)^e)$  for  some  {\bf fixed}  integers
$D,N,e$.  The exponent of $M(\Z)/\D$ depends only on the integers $D,N,e$
and not on $a,b$).

\end{corollary}

\begin{proof} This  is immediate from  Corollary \ref{rankoneisotropy}
and Proposition \ref{dirichlet}.
\end{proof}

\begin{notation}  We set  up some  notation for  another corollary  to
Proposition \ref{dirichlet}. Let  $\theta: H=SL_2 \ra G$  be as before
with $\theta  (U_H^{\pm})\subset U^{\pm}$.  Denote by $C$  and $C_H$
the congruence subgroup kernels of  $G$ and $H$ respectively. Then $C$
is the inverse limit
\[C  =\varprojlim \widehat{\G(m)/E(m)}\]  where $\G(m)=G(m\Z)$  is the
principal  congruence  subgroup of  level  $m$,  $E(m)$ is  the  normal
subgroup in $G(\Z)$  generated by $U^{\pm}(m\Z)$ and  the roof denotes
the  profinite  completion. Since  the  group  $M$  is assumed  to  be
abelian,  it  is a  torus  and  by a  theorem  of  Chevalley, has  the
congruence subgroup property. We can then show without too much
difficulty that 
\[ C=  \varprojlim \widehat{\G(m)/F(m)}\]  where $F(m)$ is  the normal
subgroup of  $G(\Z)$ generated by  $P^{\pm}(m\Z)$. Denote by  $C'$ the
image under the  induced map (still denoted by $\theta  $) of $C_H$ in
$C$.

\end{notation}

\begin{corollary}  \label{rankonedirichlet}   There  exists   a  fixed
infinite subgroup $\D \subset M(\Z)$  such that $\D$ acts trivially on
the image $C'$  of the congruence subgroup kernel of  $SL_2$ under the
map $\theta : SL_2 \ra G$.
\end{corollary}

\begin{proof}  By Proposition  \ref{dirichlet}  (and its  consequence,
namely Corollary  \ref{dirichlet1}), there exists a  fixed integer $D$
and  a {\it  fixed} infinite  subgroup  $\D$ in  $M(D\Z)$, which  acts
trivially  on $\theta  (g) F(m)  \in  G(m)/F(m)$ for  {\bf all}  $g\in
SL_2(Dm\Z)$.    Since   $C'$  is   the   inverse   limit  of   $\theta
(SL_2(Dm)/E_2(Dm))$, it  follows that $\D$  acts trivially on  $C'$ as
well.

\end{proof}

\begin{proposition} \label{C'trivial} Let $\theta :  SL_2 \ra G$ be as
in the paragraph preceding  Corollary \ref{SL2commutator} Denote again
by  $\theta$ the  map induced  at  the level  of congruence  subgroups
kernels: $\theta : C_H \ra C$ with image $C'$ say. \\

The action of $SL_2(\Q)$ on $C'$ is trivial.
\end{proposition}

\begin{proof} We  know that $\Delta  \subset M(\Z)$ acts  trivially on
$C'$. But  the group $SL_2(\Q)$ {\it  acts} on the kernel  $C'$; hence
the  commutator   $[\Delta,SL_2(\Q)]$  acts  trivially  on   $C'$.  By
Corollary  \ref{SL2commutator},  $SL_2(\Q)\subset [\Delta,  SL_2(\Q)]$
and hence acts trivially on $C'$.
\end{proof}

\subsection{Proof of Theorem \ref{maintheorem} when $\Q-rank (G)=1$ and $L(\Z)$ is virtually abelian}

\begin{theorem}   \label{rankonemaintheorem}   Let  $\Q-rank   (G)=1$,
$\R-rank   (G)\geq  2$   and  assume   that  $M_0(\Z)$   is  virtually
abelian.  Then  the  congruence  subgroup kernel  $C$  is  central  in
$\widehat{G}$.
\end{theorem}

\begin{proof} Fix $X\in  \fg _\a(\Q)$ such that  over the algebraic
closure  $\overline{\Q}$, the  projection of  $X$ to  each root  space
(occurring in $\fg _\a$) of a maximal torus in $G$  containing the product
torus $MS$ is nonzero. Since $L$ (being the centraliser of the split
torus $S$) contains  this maximal torus, it follows  that the space
generated by  the conjugates  $^t(X); t  \in L(\Q)$  is all  of $\fg
_\a$.  The Lie algebra  generated by $\fg _\a$ is all  of $Lie U^+$ by
part 2  of Lemma  \ref {liealgebragenerator}.  Denote  by $exp:\fg
_\a \ra U^+ \subset G$ the  exponential map on the elements (which are
of course, nilpotent matrices) of $\fg_\a$.  \\

By the  Jacobson-Morozov theorem, there exists  a homomorphism $\theta
':   H=   SL_2   \ra   G$   defined   over   $\Q$,   such   that   $u=
\theta \begin{pmatrix} 1 & 1 \\  0 & 1\end{pmatrix} \mapsto exp(X) \in
U^+$.   Denote by  $B$  (resp.   $T$) the  group  of upper  triangular
(resp. diagonal) matrices in $H$. Since $P=P_0$ is a minimal parabolic
subgroup, it follows that there exists $g_1\in G(\Q)$ which conjugates
$\theta (B)$ into $P$: $^{g_1}\theta  '(B)\subset P$. The conjugacy of
split tori in $P$ shows that  there exists $g_2\in P$ which conjugates
$\theta  '(T)$  into $S$.  The  product  $g=g_2g_1 $  then  conjugates
$\theta '(B)$ into $P$ and $\theta '(T)$ into $S$. Denote by $\theta $
the map  $h\mapsto g  \theta '(h)g^{-1}$.  Since $exp  (X)$ lies  in a
unique maximal  unipotent $\Q$-subgroup  namely $U^+$,  the inclusions
$^g(exp  (X)) \in  U^+, exp  (X)\in U^+$  imply that  $g\in P$.  Write
$g=mu$ with $m\in M_0(\Q), u \in U^+$.  \\

Now $^g(exp(X))=exp (^g (X))=exp  ^m(u(X))$. Since $u$  is unipotent,
the  element $^u(X)$  is of  the form  $X+Y$ with  $Y\in \fg  _{2\a}$.
Since $S=\theta  (T)$ acts  by the eigencharacter  $\a$ on  $^u(X)$ it
follows that $Y=0$ and $^u(X)=X \in \fg _\a$. Hence $^u(X)=X$ also has
nonzero projection to all the root  spaces of $T$, and hence the set
$^M(^u(X))$ generates $Lie (U^+)$, and the set $^M(exp (X))$ generates
$U^+(\Q)$.  Moreover,  $^\k (X)  \in \fg  _{-\a}$; hence  $^M(^\k (X))$
generates $\fg_{-\a}$ and $^M(U_H^-)$ generates $U^-$. \\

\begin{lemma} \label{U(p)andUH(p)} If $K_p$ is a compact open subgroup
in $M(\Q _p)$ then
\[^{K_p}(U^+_H(p))=^{M_p}(U_H^+)=U^+(p).\]
\end{lemma}

\begin{proof}It is enough to prove that  the $\Q _p$ Lie algebra $\fv$
generated by  $^{K_p} (U_H^+)(p))$ is  all of $\fu  = Lie  (U^+(p))$. Let
$X\in Lie (U_H^+)$  as before. If $\fv \neq \fu$  then there exists an
$m_0\in M_0(\Q)$  such that $Y=^{m_0}(X)\notin \fv$.  Suppose $\l :\fu
\ra \Q _p$ is a linear form which is nonzero on $Y$ but zero on $\fv$;
the function $\l _X: m\mapsto \l  (^m(X))$ is a polynomial function on
$M$ which  is identically zero on  the open subgroup $K_p$;  but $K_p$
being open, is Zariski dense and hence $\l_X$ is zero on all of $M_p$,
a contradiction. Therefore, $\fv =\fu$.
\end{proof}

We can now complete the proof of centrality for $G$. \\

Since  $C'$  is  central  in  $\theta  (\widehat{SL_2})$  (Proposition
\ref{C'trivial}),  it follows  from Proposition  \ref{commutinggroups}
that $U_H^+(p)$ and $U_H^-(q)$  commute. Conjugation by elements $m\in
M(\Q)$ shows that $^m(U_H(p))$ and $^m(U_H^-(q))$ commute for all $m\in
M(\Q)$.   However, the  action  of $M(\Q)$  on  $U^+(p)$ and  $U^-(q)$
factors via the map
\[  M(\Q)\ra  M(\Q  _p)\times  M(\Q _q)  \ra  Aut  (U^+(p))\times  Aut
(U^-(q)).\]  (the   maps  given   by  $m\mapsto  (m_p,m_q)$   and  the
conjugation     by     $m$     becomes     $^m(U^+(p))=^{m_p}(U^+(p)),
^m(U^-(q))=^{m_q}(U^-(q))$.  By  a theorem of Sansuc  (see Theorem 7.9
of \cite{Pl-R}), the closure of $M(\Q)$ in $M(\Q_p)\times M(\Q _p)$ is
{\it open} of  finite index.  Hence the closure of  $M(\Q)$ contains a
subgroup of  the form $K_p\times  K_q$ with $K_p\subset M(\Q  _p), K_q
\subset M(\Q _q)$ open.   Consequently, $^{K_p}(U_H(p))$ commutes with
$^{K_q}(U_H^-(q))$.   Then   Lemma  \ref{U(p)andUH(p)}   implies  that
$U^+(p)$ and $U^-(q)$ commute .  By Proposition \ref{commutinggroups},
this implies that $C$ is central.
\end{proof}

\newpage

\section{When $\Q-rank (G)\geq 2$ or $M(\Z)$
is not virtually abelian}

\subsection{Centrality of $C/C_M'$}

Assume that $\infty - rank(G)\geq  2$.  {\it Suppose that $M_s(\Z)$ is
infinite}.  Now  consider the quotient ${\widehat  G}/C_M'$; since (by
Lemma \ref{C'central})  $C'_M$ is  a central  and compact  subgroup of
$\widehat G$, it  follows that this quotient is a  locally compact and
Hausdorff topological group. \\

As  we   already  observed   in  the  discussion   preceding  equation
(\ref{parabolicinverselimit})   the    restriction   to   $P(\Z)\simeq
U(\Z)M(\Z)$ of the topology induced on $G(\Q)$ viewed as a subgroup of
$\widehat{G}/C_M'$  is  simply  the congruence  topology.   Thus,  for
$m\geq  1$   the  groups  $P(m\Z)$   give  a  fundamental   system  of
neighbourhoods of identity of $P(\Q)\subset {\widehat G}/C_M'$. \\

Let  $F(m)$  denote  the  subgroup  of  $G(m\Z)$  normalised  by
$G(\Z)$  and  generated  as  a  normal subgroup  by  the  two  groups
$P(m\Z)$  and $P^{-}(m\Z)$.  Then  the quotient  ${\widehat G}/C_M'$
maps   onto   $G(\A_f)$    with   kernel   $C/C_M'\simeq   \varprojlim
\widehat{G(m\Z)/F(m)}$ (see equation (\ref{parabolicinverselimit}). \\

An  element  in  the   quotient  group  $G(m\Z)/F(m)$  (after  perhaps
multiplying  it on  the  right  by an  element  of  the Zariski  dense
subgroup $F(m)$)  may be written  in the  form $u^{-}p $  with $u^-\in
U^-, \quad p \in P$.  Moreover, the denominators $a$ of the entries of
the  matrices $u^{\pm}$  and $p$  are co-prime  to $m$.   From equation
(\ref{conjugationbyM}),  if  $t\in  M'(Da^N\Z)$, then  $^t  (g)=g  \in
gF(m)$ for  every $t \in M(a^N\Z)$.  On the other hand,  $M(m\Z)$ also
acts  trivially   on  $G(m)/F(m)$  since  $M(m\Z)$   is  contained  in
$F(m)$. \\

If  $(D,m)$ denotes  the  g.c.d of  $D$ and  $m$,  then $M(m\Z)$  and
$M(Da^N\Z)$ together generate  the group $M((D,m)\Z)$ ,  since $m$ and
$a$ are  co-prime (it  follows from strong  approximation that  for two
non-zero  integers $u,v$  with  g.c.d.  $w$,  the  group generated  by
$M'(u\Z)$  and $M'(v\Z)$  is  all of  $M'(w\Z)$);  therefore the  {\it
infinite   group}  $M'(D\Z)\subset   M'(D,m)\Z)$  acts   trivially  on
$G(m\Z)/F(m)$ and hence on $C/C_M$. \\

Since all of $G(\Q)$ operates on the quotient $C/C'_M$ (recall that we
have  proved that  $C'_M$  is  centralised by  $G(\Q)$  and hence,  in
particular, is stable  under the action of $G(\Q)$),  and the infinite
group  $M'(D\Z)$  acts trivially,  it  follows  by the  simplicity  of
$G(\Q)$ modulo its centre, that $G(\Q)$ acts trivially; hence $C/C'_M$
is centralised by ${\widehat G}/C'_M$.  \\

\subsection{Centrality of $C$}

Now let $g\in G(\Q)$ and $c\in  C$. From the preceding subsection, $g$
acts   trivially    on   $C/C'_M$    and   $C'_M$   is    central   in
$\widehat{G}$.  Then $\psi  (g)=gcg^{-1}c^{-1}$  is in  $C_M'$ and  is
central in  $\widehat G$.  Hence  it follows that  $\psi (g_1g_2)=\psi
(g_1)\psi (g_2)$.  Thus $\psi :G(\Q)\ra  C_M'$ is a  homomorphism into
the abelian group  $C_M'$.  Since $G(\Q)$ is simple  modulo centre, it
follows  that $\psi  $ is  trivial and  hence that  $C$ is  central in
$\widehat G$. \\

This  and  Theorem  \ref{rankonemaintheorem}  together  prove  Theorem
\ref{maintheorem} in all cases.

\end{document}